\documentclass[12pt,reqno]{amsart}

\usepackage{amsfonts,amssymb}
\usepackage{amsmath}
\usepackage{amsthm,amscd,latexsym}
\usepackage{color}

\newtheorem{theorem}{Theorem}[section]
\newtheorem{corollary}[theorem]{Corollary}
\newtheorem{lemma}[theorem]{Lemma}

\newtheorem{Definition}[theorem]{Definition}

\newtheorem{Example}[theorem]{Example}
\newtheorem{Remark}[theorem]{Remark}

\numberwithin{equation}{section}

\newenvironment{remark}{\begin{Remark}\begin{em}}{\end{em}\end{Remark}}

\newenvironment{definition}{\begin{Definition}\begin{em}}{\end{em}\end{Definition}}

\def\bP{\mathbb{P}}
\def\cP{\mathcal{P}}

\def\bH{\mathbb{H}}

\def\<{\langle}
\def\>{\rangle}
\def\bM{\mathbb{M}}

\def\bR{\mathbb{R}}

\DeclareMathOperator{\tr}{tr}
\def\argmin{\mathop{\mathrm{arg\,min}}}
\begin{document}
\allowdisplaybreaks

\title[Geometric mean flows and the Cartan barycenter]
{Geometric mean flows and the Cartan barycenter on the Wasserstein space
over positive definite matrices}
\author[Hiai and Lim]{Fumio Hiai and Yongdo Lim}
\address{Tohoku University (Emeritus), Hakusan 3-8-16-303, Abiko 270-1154, Japan}\email{fumio.hiai@gmail.com}
\address{Department of Mathematics, Sungkyunkwan University, Suwon 440-746, Korea} \email{ylim@skku.edu}
\date{\today}
\maketitle

\begin{abstract} We introduce a class of flows on the Wasserstein space
of probability measures with finite first moment on the Cartan-Hadamard
Riemannian manifold of positive definite matrices, and consider the problem
of differentiability of the corresponding Cartan barycentric trajectory.
As a consequence we have a version of Lie-Trotter formula and a related
unitarily invariant norm inequality. Furthermore, a fixed point theorem
related to the Karcher equation and the Cartan barycentric trajectory is
also presented as an application.
\end{abstract}

\medskip
\noindent \textit{2010 Mathematics Subject Classification}. 15A42,
47A64, 47B65, 47L07

\noindent \textit{Key words and phrases.} Positive definite matrix,
Probability measure, Riemannian trace metric, Cartan barycenter,
Wasserstein distance, Lie-Trotter formula

\section{Introduction and main theorem}
Let $\bP_m$ be the set of $m\times m$ positive definite matrices,
which is a smooth Riemannian manifold with the \emph{Riemannian trace
metric} $ \<X,Y\>_A=\tr A^{-1}XA^{-1}Y,$ where $A\in {\Bbb P}_{m}$
and $X,Y\in {\Bbb H}_{m},$ the Euclidean space of $m\times m$
Hermitian matrices equipped with the inner product $\langle
X,Y\rangle={\mathrm{tr}}\,XY $. Then ${\Bbb P}_{m}$ is a
Cartan-Hadamard Riemannian manifold, a simply connected complete
Riemannian manifold with non-positive sectional curvature (the
canonical $2$-tensor is non-negative). The Riemannian distance
between $A,B\in\bP_m$ with respect to the above metric is given by $
d(A,B)=\|\log A^{-1/2}BA^{-1/2}\|_2$, where $\|X\|_2=(\tr
X^2)^{1/2}$ for $X\in\bH_m$, and the unique (up to parametrization)
geodesic joining $A$ and $B$ is given as the curve of
\emph{weighted geometric means}
\begin{equation}\label{w-g-mean}
t\in[0,1]\ \,\longmapsto\ \,A \#_{t} B :=
A^{\frac{1}{2}}(A^{-\frac{1}{2}}BA^{-\frac{1}{2}})^{t}A^{\frac{1}{2}}.
\end{equation}

\par
Let $\cP(\bP_m)$ denote the set of all probability measures on the
Borel sets of $\bP_m$, and $\cP^1(\bP_m)$ be the set of
$\mu\in\cP(\bP_m)$ with finite first moment, i.e., for some
(equivalently for all) $Y\in\bP_m$,
$\int_{\bP_m}d(X,Y)\,d\mu(X)<\infty.$ For $\mu\in\cP^1(\bP_m),$ the
{\it Cartan barycenter} $G(\mu)\in\bP_m$ is uniquely defined as
$$
G(\mu)=\argmin_{Z\in\bP_m}\int_{\bP_m}\bigl[d^2(Z,X)-d^2(Y,X)\bigr]\,d\mu(X)
$$
independently of the choice of a fixed $Y\in\bP_m$ (see \cite{St}).
For every $\mu\in\cP^1(\bP_m)$, $X=G(\mu)$ is characterized by the
{\it Karcher equation}
\begin{equation}\label{Karcher}
\int_{\bP_m}\log X^{-1/2}AX^{-1/2}\,d\mu(A)=0,
\end{equation}
which is equivalent to the gradient zero equation for the function
$Z\mapsto\int_{\bP_m}\bigl[d^2(Z,X)-d^2(Y,X)\bigr]\,d\mu(X)$ on
$\bP_m$. See \cite[Theorem 3.1]{HL}.

\par
When $A_1,\dots,A_n\in\bP_m$ and $w=(w_1,\dots,w_n)$ is a weight
vector (i.e., $w_j\ge0$, $\sum_{j=1}^nw_j=1$), we denote by
$G_w(A_1,\dots,A_n)$ the Cartan mean $G(\mu)$ of a finitely
supported measure $\mu=\sum_{j=1}^{n}w_j\delta_{A_{j}}$, where
$\delta_A$ is the point measure of mass $1$ (or the Dirac mass) at
$A \in \bP_m$. In particular, $G_w(A,B)$ with $w=(1-t,t)$ for $0\le t\le1$
coincides with the weighted geometric mean in \eqref{w-g-mean}.
For $n > 2$ we have no such formula, and properties of $G_w(A_{1},
\dots, A_{n})$ have to be established by indirect arguments. The
multivariate mean $G_w(A_1,\dots,A_n)$ has been the subject of
intensive study in the past ten years, e.g., \cite{Mo,BH,LL1,BK,LP,Ya1}.

\par
We now introduce a class of flows induced by the weighted geometric mean map
on the probability measure space ${\mathcal P}^1({\Bbb P}_{m})$.

\begin{definition}\label{D:flow}\rm
For $X\in\bP_m$,
$\mu\in{\mathcal P}^{1}(\bP_m)$ and $t\in\bR$, define
$X\#_{t}\mu\in {\mathcal P}^{1}(\bP_m)$ by
\begin{equation}\label{flow}
X\#_{t}\mu:=(f_t)_{*}\mu,
\end{equation}
i.e., the push-forward of $\mu$ by the
homeomorphic map $f_t:\bP_{m}\to \bP_{m}$ defined by
$f_t(A):=X\#_{t}A$, where we use the notation $\#_t$ given in
\eqref{w-g-mean} without restricting to $0\le t\le1$ (indeed,
the expression in \eqref{w-g-mean} is meaningful for all $t\in\bR$).
We also define the Cartan barycentric trajectory of \eqref{flow} by
\begin{eqnarray}\label{beta}
\beta(t)=\beta_X^\mu(t):=G(X\#_{t}\mu),\qquad t\in\bR.
\end{eqnarray}
\end{definition}

The one-parameter family $X\#_t\mu$ provides a flow on $\cP^1(\bP_m)$
and is also considered as a $\bP_m$-valued Markov process (see
Theorem \ref{T:flow} and Remark \ref{R:flow} for more details).

\par
The main result of the paper is the following:

\begin{theorem}\label{T:main}
Let $X\in {\Bbb P}_{m}$ and $\mu\in {\mathcal P}^{1}({\Bbb P}_{m})$.
Then the map $\beta:\bR\to {\Bbb P}_{m}$ defined by \eqref{beta} is locally
Lipschitz continuous on $\bR$ and differentiable at $t=0$ with
\begin{eqnarray}\label{X-int}
\beta'(0)=X^{1/2}\biggl(\int_{\bP_{m}}\log
X^{-1/2}AX^{-1/2}\,d\mu(A)\biggr)X^{1/2}.
\end{eqnarray}
\end{theorem}

The proof of the theorem will be presented in Section 3. The theorem
has an important consequence on the Lie-Trotter formula for the Cartan
barycenter, as shown in the rest of this introductory section.

\par
For general square matrices $X$ and $Y,$ the well-known {\it
Lie-Trotter formula} expresses
$$\lim_{n\to\infty}(e^{X/n}e^{Y/n})^n=e^{X+Y}.$$ The symmetric form
with a continuous parameter is also well-known as
$$\lim_{t\to0}(A^{t/2}B^tA^{t/2})^{1/t}=\exp(\log A+\log B)$$ for
$A,B\in\bP_m$. This formula has also been known in many other
situations; for example, see \cite{HiP0,Hi,Fu,AKL,BG} for
$A\#_\alpha B$ and other means. The Lie-Trotter formula for the
Cartan mean $G_w(A_{1},\dots,A_{n})$ of a finite number of
$A_j\in\bP_m$ (or a finitely supported measure
$\mu=\sum_{j=1}^nw_j\delta_{A_j}$) is
\begin{equation}\label{Lie-Trotter1}
{\underset{t\to 0}{\lim}}\, G_w(
A_1^t,\ldots,A_m^t)^{\frac{1}{t}}=\exp\left( \sum_{j=1}^{n}w_j\log
A_{j}\right),
\end{equation}
as given in \cite{FFS,HiP}. In \cite{HL}, the authors have extended
this Lie-Trotter formula for a certain sub-class of $\cP^1(\bP_m)$
in such a way that
\begin{equation}\label{Lie-Trotter2}
\lim_{t\to0}G(\mu^t)^{\frac{1}{t}}=\exp\int_{\bP_m}\log A\,d\mu(A)
\end{equation}
for any $\mu\in\cP(\bP_m)$ satisfying
$\int_{\bP_m}(\|A\|+\|A^{-1}\|)^r\,d\mu(A)<\infty$ for some $r>0$.
Here, $\|A\|$ denotes the operator norm of $A$, while any two norms on
$\bH_m$ are equivalent due to finite dimensionality.

The action of $t$-th power $\mu^t$ on $\cP(\bP_m)$ is defined by the
push-forward measure of $\mu$ by the matrix $t$-th power  $A\mapsto
A^t$ on $\bP_m$, that is,
\begin{eqnarray}\label{E:pow}
\mu^t({\mathcal O})=\mu\bigl(\bigl\{A^{\frac{1}{t}}: A\in {\mathcal
O}\bigr\}\bigr)
\end{eqnarray}
for any Borel set $\mathcal O\subset\bP_m$, which is indeed
comparable to the case in \eqref{Lie-Trotter1} since
$\mu^{t}=\sum_{j=1}^{n}w_j\delta_{A_{j}^{t}}$ for
$\mu=\sum_{j=1}^{n}w_j\delta_{A_{j}}$. When $X$ is the identity
matrix $I=I_m$, we have $\beta(t)=G(\mu^t)$ with $\beta(0)=I$.
Theorem \ref{T:main} implies that $ \beta(t)=I+t\beta'(0)+o(t)$ so
that
$$
{1\over t}\log\beta(t)=\beta'(0)+{o(t)\over t} \ \,\longrightarrow\
\,\beta'(0)\qquad as\ \ t\to0.
$$
Therefore,
$$
\lim_{t\to0}G(\mu^t)^{\frac{1}{t}}=\lim_{t\to0}\beta(t)^{\frac{1}{t}}=\exp\beta'(0)
=\exp\int_{\bP_n}\log A\,d\mu(A).
$$
This provides the following extension of the above Lie-Trotter formula
to the most general case of $\mu\in\cP^1(\bP_m)$.

\begin{corollary}\label{C:LT}
The formula \eqref{Lie-Trotter2} holds true for every $\mu\in\cP^1(\bP_m)$.
\end{corollary}

It turns out \cite[Corollary 4.5]{HL} that
$\big|\big|\big|G(\mu^t)^{1\over t}\big|\big|\big|$ is increasing as
$t\searrow0$ for any unitarily invariant norm $|||\cdot|||$. As a
byproduct of Corollary \ref{C:LT} we have:

\begin{corollary}\label{C:norm}
Let $\mu\in\cP^1(\bP_m)$. Then for every unitarily invariant norm
$|||\cdot|||$ and for every $t>0$,
\begin{equation*}
\big|\big|\big|G(\mu^{-t})^{-{1\over t}}\big|\big|\big|
=\big|\big|\big|G(\mu^t)^{1\over t}\big|\big|\big|
\le\bigg|\bigg|\bigg|\exp\int_{\bP_m}\log
X\,d\mu(X)\bigg|\bigg|\bigg|,
\end{equation*}
and $\big|\big|\big|G(\mu^t)^{1\over t}\big|\big|\big|$ increases to
$\big|\big|\big|\exp\int_{\bP_m}\log X\,d\mu(X)\big|\big|\big|$ as
$t\searrow0$.
\end{corollary}

\section{Geometric mean flows on the probability measure space}

Let $X\in\bP_m$ and $\mu \in{\mathcal P}^{1}(\bP_m)$. For every $t\in\bR$,
define $X\#_t\mu$ as in Definition \ref{D:flow}, that is, $X\#_{t}\mu=(f_t)_{*}\mu$
is the push-forward of $\mu$ by $f_t:\bP_m\to\bP_m$, $f_t(A)=X\#_{t}A$.

\begin{lemma}\label{L:2.1}
We have $X\#_t\mu\in\cP^1(\bP_m)$ for every $t\in\bR$.
\end{lemma}

\begin{proof}
It is immediate to see that
\begin{equation}\label{E:2.1}
\|\log A\|=\log\max\{\|A\|,\|A^{-1}\|\},\qquad A\in\bP_m.
\end{equation}
When $t>0$, we have
$$
\|X\#_tA\|\le\|X\|\,\|X^{-1/2}AX^{-1/2}\|^t\le\|X\|\,\|X^{-1}\|^t\|A\|^t,
$$
$$
\|(X\#_tA)^{-1}\|=\|X^{-1/2}(X^{1/2}A^{-1}X^{1/2})^tX^{-1/2}\|
\le\|X^{-1}\|\,\|X\|^t\|A^{-1}\|^t.
$$
Therefore, by \eqref{E:2.1} we have
$$
\|\log(X\#_tA)\|\le(1+t)\|\log X\|+t\|\log A\|,\qquad A\in\bP_m,
$$
which implies that $X\#_t\mu\in\cP^1(\bP_m)$ since $d(X,I)=\|\log
X\|_{2}\leq m\|\log X\|$ for all $X\in {\Bbb P}_{m}$ and
$$
\int_{\bP_m}\|\log A\|\,d(X\#_t\mu)(A)
=\int_{\bP_m}\|\log(X\#_tA)\|\,d\mu(A)<\infty.
$$
When $t<0$, the argument is similar since
$X\#_tA=X^{1/2}(X^{1/2}A^{-1}A^{1/2})^{-t}X^{1/2}$.
\end{proof}

Note that $I\#_{t}\mu=\mu^{t}$, where $\mu^t$ is defined in \eqref{E:pow},
$ X\#_{0}\mu=\delta_{X}$ and $X\#_{1}\mu=\mu.$ When $t\ne0$, since
$X\#_{t}Z=A$ if and only if $Z=X\#_{1/t}A$, i.e., $f_t^{-1}=f_{1/t}$, we see that
\begin{displaymath}
\displaystyle (X\#_{t}\mu)(\mathcal{O}) = \mu( \{ X\#_{1/t}A:A\in \mathcal{O} \})
\end{displaymath}
for any Borel set $\mathcal O\subset\bP_m$. Moreover, note that if
$\mu=\frac{1}{n}\sum_{j=1}^{n}\delta_{A_{j}}$, then
$X\#_{t}\mu=\frac{1}{n}\sum_{j=1}^{n}\delta_{X\#_{t}A_{j}}$.

The \emph{$1$-Wasserstein distance} $d_1^W$ on $\cP^1(\bP_m)$ is defined by
$$
d_1^{W}(\mu,\nu):=\inf_{\pi\in\Pi(\mu,\nu)}\int_{\bP_m\times\bP_m}d(X,Y)\,d\pi(X,Y),
\qquad\mu,\nu\in\cP^1(\bP_m),
$$
where $\Pi(\mu,\nu)$ is the set of all couplings for $\mu,\nu$,
i.e., $\pi\in\cP(\bP_m\times\bP_m)$ whose marginals are $\mu$ and $\nu$.
Recall (see \cite{St}) that $\cP^1(\bP_m)$ is a complete metric space with
the metric $d_1^W$ and that the set $\cP_0(\bP_m)$ of uniform probability measures
with finite support (i.e., the measures of the form ${1\over n}\sum_{j=1}^n\delta_{A_j}$)
is dense in $\cP^1(\bP_m)$. An important fact called the \emph{fundamental contraction
property} in \cite{St} (also \cite[Theorem 2.3]{HL}) is that the
Cartan barycenter $G:\cP^1(\bP_m)\to {\bP_m}$ is a Lipschitz map
with Lipschitz constant $1$; namely, for every
$\mu,\nu\in\cP^1(\bP_m)$,
\begin{eqnarray}\label{contract}
d(G(\mu),G(\nu))\leq d_{1}^{W}(\mu,\nu).
\end{eqnarray}

\par
The next lemma will play a role, which was given in \cite[Lemma 2.2]{LL5} in
a more general setting.

\begin{lemma}\label{L:2.2}
Let $f:\bP_m\to\bP_m$ be a Lipschitz map with Lipschitz constant $C$. Then
the push-forward map $f_*:\cP^1(\bP_m)\to\cP^1(\bP_m)$, $\mu\mapsto f_*\mu$,
is Lipschitzian with respect to $d_1^W$ with Lipschitz constant $C$.
\end{lemma}

\begin{lemma}\label{L:2.3}
For every $\mu,\nu\in {\mathcal P}^{1}({\Bbb P}_{m})$ and $t,s\in
[0,1],$
$$
d_{1}^{W}(X\#_{t}\mu,Y\#_{s}\nu)\leq (1-t)d(X,Y)+td_{1}^{W}(\mu,\nu)
+|t-s| d_{1}^{W}(\delta_{Y}, \nu).
$$
\end{lemma}

\begin{proof} It is known (see \cite{Bha}) that
\begin{align}
d(A\#_{t}B, C\#_{t}D)&\leq(1-t)d(A,C)+td(B,D), \label{E:3.1}\\
d(A\#_{t}B, A\#_{s}B)&=|t-s|d(A,B),\qquad t,s\in [0,1]. \nonumber
\end{align}
By the triangular inequality, for every $t,s\in[0,1]$,
\begin{align*}
d(A\#_t B, C\#_s D)&\leq d(A\#_{t}B,C\#_{t}D)+d(C\#_{t}D,C\#_{s}D)\\
&\leq(1-t)d(A,C)+td(B,D)+|t-s|d(C,D).
\end{align*}

For $\mu=\frac{1}{n}\sum_{j=1}^{n}\delta_{A_{j}}$,
$\nu=\frac{1}{n}\sum_{j=1}^{n}\delta_{B_{j}}$ in ${\mathcal
P}_{0}({\Bbb P}_{m})$,  it is known (see Introduction of \cite{Vi1}) that
$$
d^W_1(\mu,\nu)=\min_{\sigma\in S_{n}}\frac{1}{n}\sum_{j=1}^{n}
d(A_{j},B_{\sigma(j)}),
$$
where $S_n$ is the permutation group on $\{1,\ldots,n\}$. Therefore,
for every $t,s\in[0,1]$ we find a $\sigma\in S_n$ so that
\begin{align*}
d_1^W(X\#_t\mu,Y\#_s\nu)
&={1\over n}\sum_{j=1}^nd(X\#_tA_j,Y\#_sB_{\sigma(j)}) \\
&\le{1\over n}\sum_{j=1}^n\bigl[(1-t)d(X,Y)+td(A_j,B_{\sigma(j)})
+|t-s|d(Y,B_{\sigma(j)})\bigr] \\
&\le(1-t)d(X,Y)+td_1^W(\mu,\nu)+|t-s|d_1^W(\delta_Y,\nu).
\end{align*}
Hence the required inequality holds for all
$\mu,\nu\in\cP_0(\bP_m)$. Since $d(X\#_tA,X\#_tB)\le td(A,B)$ by
\eqref{E:3.1}, we see by Lemma \ref{L:2.2} that $\mu\mapsto
X\#_t\mu$ is Lipschitzian with Lipschitz constant $t$. Since
$\cP_0(\bP_m)$ is dense in $\cP^1(\bP_m)$, the result follows.
\end{proof}

\begin{theorem}\label{T:flow}
For each $X\in {\Bbb P}_{m},$ the map
$\Phi_{X}:{\Bbb R}\times {\mathcal P}^{1}({\Bbb P}_{m})\to {\mathcal
P}^{1}({\Bbb P}_{m})$ defined by
$$\Phi_{X}(t,\mu)=X\#_t\mu$$
is a continuous flow satisfying
\begin{equation}\label{Phi-compos}
\Phi_X(ts,\mu)=\Phi_X(s,\Phi_X(t,\mu)),\qquad t,s\in\bR.
\end{equation}
Moreover, for a fixed $\mu\in {\mathcal P}^{1}({\Bbb P}_{m})$, the map
$t\in\bR\mapsto X\#_t\mu\in\cP^1(\bP_m)$ is locally Lipschitz
continuous with respect to $d_1^W$, that is, for every $T>0$ there
exists a constant $C_T>0$ such that
$$
d_1^W(X\#_t\mu,X\#_s\mu)\le C_T|t-s|,\qquad t,s\in[-T,T].
$$
\end{theorem}

\begin{proof} It is immediate to see that $X\#_s(X\#_tA)=X\#_{st}A$
for every $t,s\in\bR$, which yields
\begin{equation}\label{compos}
X\#_{st}\mu=X\#_s(X\#_t\mu),\qquad t,s\in\bR.
\end{equation}
This is nothing but \eqref{Phi-compos}. Continuity follows
from Lemma \ref{L:2.3}.

Let $\mu\in {\mathcal P}^{1}({\Bbb P}_{m})$ be fixed. Lemma
\ref{L:2.3} shows in particular that $d_1^W(X\#_t\mu,X\#_s\mu)\le
C_1|t-s|$ for every $t,s\in[0,1]$ with $C_1:=d_1^W(\delta_X,\mu)$.
When $t,s\in[0,1]$, since $X\#_{-t}A=X(X\#_tA)^{-1}X$ and
$X\#_{-s}A=X(X\#_sA)^{-1}X$, we have
$$
d(X\#_{-t}A,X\#_{-s}B)=d(X\#_tA,X\#_sB),\qquad A,B\in\bP_m,
$$
which immediately gives
$$
d_1^W(X\#_{-t}\mu,X\#_{-s}\mu)=d_1^W(X\#_t\mu,X\#_s\mu)\le C_1|t-s|.
$$
Moreover,
\begin{align*}
d_1^W(X\#_t\mu,X\#_{-s}\mu)&\le d_1^W(X\#_t\mu,\delta_X)+d_1^W(\delta_X,X\#_{-s}\mu) \\
&\le C_1t+C_1s=C_1|t-(-s)|.
\end{align*}
Hence the result holds for $T=1$.

For any $T>0$ and $t,s\in[-T,T]$ write $t=t'T$ and $s=s'T$ with
$t',s'\in[-1,1]$. Then by \eqref{compos} we can write
$X\#_t\mu=X\#_{t'}\mu'$ and $X\#_s\mu=X\#_{s'}\mu'$ with
$\mu':=X\#_T\mu$, By the above case with $\mu'$ in place of $\mu$ we
have
$$
d_1^W(X\#_t\mu,X\#_s\mu)\le C_1'|t'-s'|={C_1'\over T}\,|t-s|
$$
for some constant $C_1'$. Hence the result follows with
$C_T:=C_1'/T$.
\end{proof}

\begin{remark}\label{R:flow}
Theorem \ref{T:flow} says that $\Phi_X(\mu,t)=X\#_t\mu$ ($t\in\bR$) is a multiplicative
$\bR$-flow on $\cP^1(\bP_m)$. Modifying as $\Psi_X(\mu,t):=X\#_{e^{-t}}\mu$ ($t\ge0$),
we have an additive $\bR_+$-flow on $\cP^1(\bP_m)$ starting at $\mu$ ($t=0$) and attracted
to $\delta_X$ (as $t\to\infty$). This flow is also considered as a $\bP_m$-valued
Markov stochastic process $X_t(A):=X\#_{e^{-t}}A$ (with smooth sample paths) on the
probability space $(\bP_m,\mu)$.
\end{remark}

\section{Proof of Theorem \ref{T:main}}
In the following  we fix $X\in {\Bbb P}_{m}$ and $\mu\in {\mathcal
P}^1({\Bbb P}_{m}).$ For notational simplicity we write
$X_t\in\bP_n$ for $\beta(t)=G(X\#_t\mu)$ (with $X_0=X$), which is
uniquely characterized by the Karcher equation (see \eqref{Karcher})
$$
\int_{\bP_n}\log X_t^{-1/2}AX_t^{-1/2}\,d(X\#_t\mu)(A)=0,
$$
that is,
\begin{equation}\label{Karcher-flow}
\int_{\bP_n}\log X_t^{-1/2}(X\#_tA)X_t^{-1/2}\,d\mu(A)=0.
\end{equation} We set $\mu_{X}:=(g_X)_{*}\mu$, where
$g_X:{\Bbb P}_{m}\to {\Bbb P}_{m}$ is defined by
$g_X(A):=X^{-1/2}AX^{-1/2}$. (Note that $\mu_X$ is $M.\mu$ with
$M=X^{-1/2}$ in the notation in \cite{KLL}.) Moreover, let
$$\mu_X^t:=(\mu_X)^t,$$ where  the action of $t$-th power $\mu^t$ on
$\cP(\bP_m)$ is defined by the push-forward measure of $\mu$ by the
matrix $t$-th power  $A\mapsto A^t$ on $\bP_m$, that is,
$\mu^t=I\#_{t}\mu.$

\begin{lemma}\label{L-3.3}
For  $X\in\bP_m$, $\mu\in\cP^1(\bP_m)$ and $t\in\bR$. Then
$X\#_t\mu=(\mu_X^t)_{X^{-1}}$ and
\begin{equation}\label{beta-X-G}
\beta(t)=X^{1/2}G(\mu_X^t)X^{1/2},\qquad t\in\bR.
\end{equation}
\end{lemma}

\begin{proof}
Since $f_t(A)=X^{1/2}(X^{-1/2}AX^{-1/2})^tX^{1/2}$, we have
$f_t=g_{X^{-1}}\circ h_t\circ g_X$, where $h_t(A):=A^t$. Therefore,
\begin{align*}
X\#_t\mu&=(g_{X^{-1}}\circ h_t\circ g_X)_*\mu \\
&=(g_{X^{-1}}\circ h_t)_*\mu_X
=(g_{X^{-1}})_*\mu_X^t=(\mu_X^t)_{X^{-1}}.
\end{align*}
Now, we recall (see \cite{KLL}) that the Cartan barycenter has the
invariance property $G((g_X)_*\mu)=g_X(G(\mu))$, i.e.,
$G(\mu_X)=X^{-1/2}G(\mu)X^{-1/2}$. Hence \eqref{beta-X-G} follows.
\end{proof}

To prove Theorem \ref{T:main}, we may and do assume that $X=I$ from
(\ref{beta-X-G}). In this case, $X_t=G(I\#_t\mu)=G(\mu_{I}^t)$ (with
$X_0=I$) and \eqref{X-int} is simply $\beta(0)=\int_{\bP_m}\log
X\,d\mu(X)$.

\begin{lemma}\label{L:3.1}
For any $T>0$ there exists a constant $K_T>0$ such that for every
$\alpha\in[-1,1]$ and every $t,s\in[-T,T]$,
$$
\|X_t^\alpha-X_s^\alpha\|\le K_T|s-t|.
$$
\end{lemma}

\begin{proof}
For any $T>0$, by Lemma \ref{L:2.3} we have
$$
d_1^W(I\#_t\mu,I\#_s\mu)\le C_T|t-s|,\qquad t,s\in[-T,T].
$$
Applying this to the fundamental contraction property \eqref{contract}
and using the exponential metric increasing property (EMI)
(see \cite[Theorem 6.1.4]{Bha})
$$
\|\log A-\log B\|_{2}\leq d(A,B),\qquad A,B\in\bP_m,
$$
we have
$$
\|\log X_t-\log X_s\|\le C_T|t-s|,\qquad T,S\in[-T,T].
$$
In particular, $\|\log X_t\|\le C_TT$ for all $t\in[-T,T]$. For any
$\alpha\in[-1,1]$ and $t,s\in[-T,T]$ we find that
\begin{align*}
\|X_t^\alpha-X_s^\alpha\|&=\|\exp(\alpha\log X_t)-\exp(\alpha\log X_s)\| \\
&\le\sum_{k=1}^\infty{\|(\alpha\log X_t)^k-(\alpha\log X_s)^k\|\over k!} \\
&\le\sum_{k=1}^\infty{\|(\log X_t)^k-(\log X_s)^k\|\over k!} \\
&\le\sum_{k=1}^\infty{kC_T(C_TT)^{k-1}\over k!}\,|t-s|=K_T|t-s|,
\end{align*}
where $K_T:=C_Te^{C_TT}$.
\end{proof}

\begin{lemma}\label{L:3.2}
There exists a constant $C>0$ such that
$$
{1\over|t|}\|\log X_t^{-1/2}A^tX_t^{-1/2}\|\le C+\|\log A\|
$$
for every $t\in[-1,1]\setminus\{0\}$ and every $A\in\bP_n$.
\end{lemma}

\begin{proof} From $X_{-t}=G(\mu^{-t})=G((\mu^{-1})^t)$, we may
assume that $t\in(0,1]$. Since
$$
\|X_t^{-1/2}A^tX_t^{-1/2}\|\le\|X_t^{-1}\|\,\|A\|^t,\qquad
\|X_t^{1/2}A^{-t}X_t^{1/2}\|\le\|X_t\|\|A^{-1}\|^t,
$$
we have
\begin{align*}
{1\over t}\log\|X_t^{-1/2}A^tX_t^{-1/2}\|
&\le{1\over t}\log\|X_t^{-1}\|+\log\|A\|, \\
{1\over t}\log\|X_t^{1/2}A^{-t}X_t^{1/2}\|
&\le{1\over t}\log\|X_t\|+\log\|A^{-1}\|.
\end{align*}
Let $M_t:=X_t-I$ and $M_t':=X_t^{-1}-I$. Then we have
\begin{align*}
{1\over t}\log\|X_t^{-1}\|
&={1\over t}\log\|I-X_t^{-1/2}M_tX_t^{-1/2}\| \\
&\le{1\over t}\log(1+\|X_t^{-1}\|\,\|M_t\|)\le\|X_t^{-1}\|\,{\|M_t\|\over t}, \\
{1\over t}\log\|X_t\|
&={1\over t}\log\|I-X_t^{1/2}M_t'X_t^{1/2}\| \\
&\le{1\over t}\log(1+\|X_t\|\,\|M_t'\|)\le\|X_t\|\,{\|M_t'\|\over t}.
\end{align*}
Note here that $\|X_t\|$, $\|X_t^{-1}\|$, $\|M_t\|/t$ and $\|M_t'\|/t$ are
all uniformly bounded for $t\in(0,1]$ by Lemma \ref{L:3.1}. Combining the above
estimates together with \eqref{E:2.1}, we find a constant $C>0$ such that
$$
{1\over t}\,\|\log X_t^{-1/2}A^tX_t^{-1/2}\|\le C+\|\log A\|,\qquad
t\in(0,1].
$$
\end{proof}

\vspace{4mm} \emph{Proof of Theorem \ref{T:main}.} For
$t\in[-1,1]\setminus\{0\}$ let $H_t:=X_t^{-1/2}-I$. We will prove
that $H_t/t$ converges as $t\to0$. Since $\|H_t/t\|$ is bounded by
Lemma \ref{L:3.1}, we may prove that a limit point of $H_t/t$ as
$t\to0$ is unique. Note that for each $A\in\bP_n$
\begin{align*}
X_t^{-1/2}A^tX_t^{-1/2} &=(I+H_t)\bigl(I+t\log A+o(t)\bigr)
(I+H_t) \\
&=I+2H_t+t\log A+o(t).
\end{align*}
Now, assume that $H_{t_k}/t_k\to L$ for a sequence
$t_k\in[-1,1]\setminus\{0\}$ with $t_k\to0$, so that
$$
{1\over t_k}\log X_{t_k}^{-1/2}A^{t_{k}}X_{t_k}^{-1/2}
=2{H_{t_k}\over t_k}
+\log A+{o(t_k)\over t_k}
\ \,\longrightarrow\ \,2L+\log A
$$
as $k\to\infty$. By \eqref{Karcher-flow} we have
\begin{equation}\label{Karcher-flow2}
\int_{\bP_n}{1\over t_k}\log
X_{t_k}^{-1/2}A^{t_{k}}X_{t_k}^{-1/2}\,d\mu(A)=0.
\end{equation}
Thanks to Lemma \ref{L:3.2}, the Lebesgue convergence theorem can be applied to
\eqref{Karcher-flow2} so that we obtain
\begin{equation*}
2L=-\int_{\bP_n}\log A\,d\mu(A).
\end{equation*}
Therefore, $L$ is a unique limit point of
$H_t/t$ as $t\to0$. This means that $t\mapsto Y_t:=X_t^{-1/2}$ is
differentiable at $t=0$ with the derivative $L$. Since
$X_t=Y_t^{-2}$, we find that $\beta(t)=X_t$ is differentiable at
$t=0$ and
\begin{align*}
\beta'(0)&=-2L=\int_{\bP_n}\log A\,d\mu(A),
\end{align*}
which is the desired conclusion (as we assumed that $X=I$). \qed

\begin{theorem}\label{T:main2}
Let $X\in {\Bbb P}_{m}, \mu\in {\mathcal P}^{1}({\Bbb P}_{m})$ and
let $\beta(t)=G(X\#_{t}\mu)$ be as in Theorem \ref{T:main}. Then
the following are equivalent:
\begin{itemize}
\item[(i)] $\beta'(0)=0$;
\item[(ii)] $X=G(\mu)$;
\item[(iii)] $X=G(X\#_{t}\mu)$ for all $t\in\bR$ $($equivalently for some $t\ne0)$;
\item[(iv)] $I=G(\mu_{X}^t)$ for all $t\in\bR$ $($equivalently for some
$t\ne0)$.
\end{itemize}
\end{theorem}

\begin{proof}
For every $\mu\in\cP^1(\bP_m)$, the Karcher equation \eqref{Karcher} is equivalent
to $\beta'(0)=0$ thanks to \eqref{X-int}. Hence we have (i)$\iff$(ii). Moreover,
we note that
\begin{align*}
\int_{\bP_m}\log X^{-1/2}AX^{-1/2}\,d(X\#_t\mu)(A)
&=\int_{\bP_m}\log X^{-1/2}(X\#_tA)X^{-1/2}\,d\mu(A)\\
&=t\int_{\bP_m}\log X^{-1/2}AX^{-1/2}\,d\mu(A) \\
&=t\int_{\bM_m}\log A\,d\mu_X(A)\\
&=\int_{\bM_m}\log A\,d\mu_X^t(A).
\end{align*}
Therefore, it immediately follows that (ii)--(iv) are equivalent.
\end{proof}

\begin{corollary} Let $\mu\in {\mathcal P}^{1}({\Bbb P}_{m})$ and
let $\beta(t):=G(G(\mu)\#_{t}\mu).$ Then $\beta$ is differentiable
at $t=0$ with $\beta'(0)=0.$
\end{corollary}

When $\mu\in\cP^2(\bP_m)$, i.e., $\mu$ has finite second moment, the
equivalence of (ii) and (iii) of Theorem \ref{T:main2} was shown in
\cite[Theorem 3.1]{KLL}. The Karcher equation or equivalently $\beta'(0)=0$
has played a crucial role in the Riemannian geometric approach of
multivariate geometric means as in \cite{Mo,LP,LL2}, which has been
extended to the Cartan barycenter in \cite{KL,KLL,HL}. For a finitely
supported measure $\mu=\frac{1}{n}\sum_{j=1}^{n}w_j\delta_{A_{j}}$,
the fixed point Cartan mean equation
$X=G(X\#_{t}\mu)=G_w(X\#_{t}A_{1},\dots,X\#_{t}A_{n})$ appeared in
\cite{LP} and \cite{LL2}. The formula \eqref{X-int} is evidently new
and deserves to receive its attention due to its relation to the
Karcher equation.

\section{Final remarks and open problems}

(1)\enspace
In the present paper, we first prove the differentiability of the Cartan
barycentric trajectory $\beta(t)$ at $t=0$ and then use it to prove the
Lie-Trotter formula for $\lim_{t\to0}G(\mu^t)^{1/t}$. One can also proceed
in the opposite way. Indeed, we have a direct proof of the Lie-Trotter formula
in Corollary \ref{C:LT}, which in turn shows Theorem \ref{T:main} immediately.
It is worth noting that the Lebesgue convergence theorem is essential in our
direct proof of \eqref{Lie-Trotter2} for $\mu\in\cP^1(\bP_m)$, as it is so in
the proof of Theorem \ref{T:main} in Section 3.

(2)\enspace We are also interested in the extension of Theorem
\ref{T:main} to any $t\in\bR$, that is, in the differentiability
problem of $\beta(t)$ and, in this case, in what is the form of
derivative $\beta'(t)$. It does not seem possible to generalize the
above proof for $\beta'(0)$ to the case for $\beta'(t)$ at $t\ne0$.
But, under a stronger assumption that
$\int_{\bP_n}(\|A\|+\|A^{-1}\|)^{2\alpha}\,d\mu(A)<\infty$ with some
$\alpha>0$, we can prove the differentiability of $\beta(t)$ for
$t\in[-\alpha,\alpha]$, though the expression of $\beta'(t)$ is much
complicated.

(3)\enspace
Given $\mu\in\cP^1(\bP_m)$, it is well-known (see \cite{LL3}) that the (Euclidean) gradient
of the function $\psi(X):={1\over2}\int_{\bP_m}\bigl[d^2(X,A)-d^2(Y,A)\bigr]\,d\mu(A)$ at
$X\in\bP_m$ is
$$
\nabla\psi(X)=X^{-1/2}\biggl(\int_{\bP_n}\log X^{1/2}A^{-1}X^{1/2}\,d\mu(A)\biggr)X^{-1/2},
$$
and the Riemannian gradient of $\psi$ at $X$ is
$\nabla^{\mathrm{Rie}}\psi(X)=X\nabla\psi(X)X$. Hence the Riemannian gradient flow on
$\bP_m$ is introduced as the solution of the Cauchy problem
$$
{dX_t\over dt}=-\nabla^{\mathrm{Rie}}\psi(X_t)
=X_t^{1/2}\biggl(\int_{\bP_m}\log X_t^{-1/2}AX_t^{-1/2}\,d\mu(A)\biggr)X_t^{1/2}
$$
with initial value $X_0=X\in\bP_m$. In \cite{LP2}, Lim and P\'alfia have discussed this
gradient flow (called an ODE flow there) and obtained its description by using the
\emph{resolvent operator} defined by
$$
J_\lambda^\mu(X):=G\biggl({\lambda\over\lambda+1}\mu+{1\over\lambda+1}\delta_X\biggr)
$$
for $\lambda\ge0$ and $X\in\bP_m$. Note that $J_\lambda^\mu(X)$ is
the Cartan barycentric trajectory of the arithmetic mean flow
$\lambda\ge0\mapsto(\lambda\mu+\delta_X)/(\lambda+1)$ on
$\cP^1(\bP_m)$. When $t=\lambda/(\lambda+1)$, from the
arithmetic-geometric mean inequality $X\#_tA\le(X+\lambda
A)/(\lambda+1)$, we can see that
$X\#_t\mu\le(\lambda\mu+\delta_X)/(\lambda+1)$ in the partial order
on $\cP(\bP_m)$ considered in \cite{KLL,HL}. By the monotonicity
property of the Cartan barycenter (see \cite[Theorem 3.2]{HL}) we
have $\beta_X^\mu(t)\le J_\lambda^\mu(X)$ for
$t=\lambda/(\lambda+1)$. It might be interesting to find more
relations of the trajectory $\beta(t)=\beta_X^\mu(t)$ with
$J_\lambda^\mu(X)$ and the gradient flow.

\subsection*{Acknowledgments}
 The work of
F.~Hiai was supported by Grant-in-Aid for Scientific Research
(C)17K05266. The work of Y.~Lim was supported by the National
Research Foundation of Korea (NRF) grant funded by the Korea
government(MEST) No.2015R1A3A2031159 and 2016R1A5A1008055.


\begin{thebibliography}{99}

\bibitem{AKL}
E. Ahn, S. Kim and Y. Lim, An extended Lie-Trotter formula and its applications,
Linear Algebra Appl. \textbf{427} (2007), 190--196.

%\bibitem{AH}
%T. Ando and F. Hiai, Log majorization and complementary Golden-Thompson type
%inequalities, Linear Algebra Appl. \textbf{197} (1994), 113--131.

%\bibitem{Ar}
%H. Araki, On an inequality of Lieb and Thirring, {\it Lett. Math. Phys.} {\bf 19}
%(1990), 167--170.

\bibitem{Bha}
R. Bhatia,
Positive definite matrices,
Princeton Series in Applied Mathematics, Princeton University Press, Princeton, NJ, 2007.

\bibitem{BG}
R. Bhatia and P. Grover, Norm inequalities related to the
matrix geometric mean, Linear Algebra Appl. \textbf{437} (2012), 726--733.

\bibitem{BH}
R. Bhatia and J. Holbrook, Riemannian geometry and matrix geometric
means, Linear Algebra Appl. \textbf{413} (2006), 594--618.

\bibitem{BK}
R. Bhatia and R. Karandikar, Monotonicity of the matrix geometric
mean, Math. Ann. \textbf{353} (2012), 1453--1467.

%\bibitem{EH}
%E. Effros and F. Hansen, Non-commutative perspectives,
%Ann. Funct. Anal. \textbf{5} (2014), 74--79.

\bibitem{FFS}
J. I. Fujii, M. Fujii, Y. Seo,
The Golden-Thompson-Segal type inequalities related to the weighted geometric mean due to
Lawson-Lim, J. Math. Inequal. \textbf{3} (2009), 511--518.

\bibitem{Fu}
T. Furuta, Convergence of logarithmic trace inequalities via generalized
Lie-Trotter formulae, Linear Algebra Appl. \textbf{396} (2005), 353--372.

\bibitem{Hi}
F. Hiai, Log-majorizations and norm inequalities for
exponential operators, in {\it Linear Operators},
J. Janas, F. H. Szafraniec and J. Zem\'anek (eds.),
Banach Center Publications, Vol. 38, 1997, pp. 119--181.

\bibitem{HL}
F. Hiai and Y. Lim, Log-majorization and Lie-Trotter formula for the
Cartan barycenter on probability measure spaces, J. Math. Anal.
Appl., \textbf{453} (2017), 195-211.

\bibitem{HiP0}
F. Hiai, D. Petz, The Golden-Thompson trace inequality is complemented,
Linear Algebra Appl. \textbf{181} (1993), 153--185.

\bibitem{HiP}
F. Hiai and D. Petz, Riemannian metrics on positive definite
matrices related to means II, Linear Algebra Appl. \textbf{436}
(2012), 2117--2136.

\bibitem{KL}
S. Kim and H. Lee, The power mean and the least squares mean of
probability measures on the space of positive definite matrices,
Linear Algebra Appl. \textbf{465} (2015), 325--346.

\bibitem{KLL}
S. Kim, H. Lee and Y. Lim, An order inequality characterizing
invariant barycenters on symmetric cones,
J. Math. Anal. Appl. \textbf{442} (2016), 1--16.

\bibitem{LL1}
J. Lawson and Y. Lim, Monotonic properties of the least squares
mean, Math. Ann. \textbf{351} (2011), 267--279.

\bibitem{LL3}
J. Lawson and Y. Lim, The least squares mean of positive Hilbert--Schmidt operators,
J. Math. Anal. Appl. \textbf{403} (2013), 365--375.

\bibitem{LL2}
J. Lawson and Y. Lim, Karcher means and Karcher equations of
positive definite operators, Trans. Amer. Math. Soc. Series B
\textbf{1} (2014), 1--22.

\bibitem{LL5}
J. Lawson and Y. Lim, Contractive barycentric maps, to appear in J.
Operator Theory.

%\bibitem{LT}
%M. Ledoux and M. Talagrand, Probability in Banach spaces.
%Isoperimetry and Process, Ergebnisse der Mathematik und ihrer
%Grenzgebiete \textbf{23} (1991), Springer, Berlin.

\bibitem{LP}
Y. Lim and M. P\'{a}lfia, Matrix power means and the Karcher mean,
J. Funct. Anal. \textbf{262} (2012), 1498--1514.

\bibitem{LP2}
Y. Lim and M. P\'{a}lfia, Existence and uniquness of the $L^1$-Karcher mean,
preprint (2017). arXiv:1703.04292 [math.FA]

\bibitem{Mo}
M.~Moakher, A differential geometric approach to the geometric mean of symmetric
positive-definite matrices, SIAM J. Matrix Anal. Appl. {\bf 26} (2005),
735--747.

\bibitem{St}
K.-T. Sturm, Probability measures on metric spaces of nonpositive
curvature. Heat kernels and analysis on manifolds, graphs, and
metric spaces (Paris, 2002), 357-390, Contemp. Math., 338, Amer.
Math. Soc., Providence, RI, 2003.

\bibitem{Vi1}
C. Villani,  Topics in Optimal Transportation, Graduate Studies in
Mathematics, Vol. 58, American Mathematical Society, Providence, RI,
2003.

\bibitem{Ya1}
T. Yamazaki, The Riemannian mean and matrix inequalities related to
the Ando-Hiai inequality and chaotic order, Operators and Matrices
\textbf{6} (2012), 577-588.

\end{thebibliography}
\end{document}